\documentclass[12pt]{amsart}
\usepackage[polish,english]{babel}
\usepackage[utf8]{inputenc}
\usepackage[T1]{fontenc}
\usepackage[nomath]{lmodern}
\usepackage{bm}
\usepackage[all]{xy}

\usepackage{amssymb,amsmath,amsthm,amsfonts}
\usepackage{thmtools,mathtools}
\usepackage{microtype}
\allowdisplaybreaks
\usepackage{enumitem} 

\usepackage[dvipsnames]{xcolor}
\usepackage{graphicx} 
\usepackage{tikz,tikz-cd,pgfplots}
\pgfplotsset{compat=newest}

\usepackage{caption} 
\usepackage{fmtcount} 
\usepackage[breaklinks,pdfencoding=auto,psdextra]{hyperref}
\hypersetup{
	colorlinks,
	linkcolor={black},
	citecolor={blue},
	urlcolor={red}
}
\usepackage[nameinlink,capitalise]{cleveref}
\usepackage{comment} 

\addto\extrasenglish{}
\addto\extrasenglish{}
\addto\extrasenglish{\def\equationautorefname~#1\null{\textnormal{(#1)}\null}}
\declaretheorem[name=Theorem,refname={Theorem},style=plain,numberwithin=section]{theorem}
\declaretheorem[name=Theorem,refname={Theorem},style=plain,numbered=no]{theorem*}
\declaretheorem[name=Proposition,refname={Proposition},style=plain,sibling=theorem]{proposition}
\declaretheorem[name=Proposition,refname={Proposition},style=plain,numbered=no]{proposition*}
\declaretheorem[name=Lemma,refname={Lemma},style=plain,sibling=theorem]{lemma}
\declaretheorem[name=Strategy,refname={Strategy},style=plain,sibling=theorem]{strategy}
\declaretheorem[name=Lemma,refname={Lemma},style=plain,numbered=no]{lemma*}

\declaretheorem[name=Definition,refname={Definition},style=definition,numbered=no]{definition*}
\declaretheorem[name=Remark,refname={Remark},style=definition,sibling=theorem]{remark}
\declaretheorem[name=Remark,refname={Remark},style=remark,numbered=no]{remark*}

\declaretheorem[name=Example,refname={Example},style=definition,numbered=no]{example*}
\declaretheorem[name=Corollary,refname={Corollary},style=plain,sibling=theorem]{corollary}
\declaretheorem[name=Corollary,refname={Corollary},style=plain,numbered=no]{corollary*}
\declaretheorem[name=Preliminaries,refname={Preliminaries},style=definition,numbered=no]{Preliminaries*}
\declaretheorem[name=Acknowledgements,refname={Acknowledgements},style=definition,numbered=no]{Acknowledgements*}

\crefname{theorem}{Theorem}{Theorems}
\crefname{lemma}{Lemma}{Lemmas}
\crefname{equation}{}{}
\crefformat{equation}{(#2#1#3)}

\def\CC{{\mathbb{C}}}

\def\bP{{\mathbf{P}}}

\def\fsl{{\mathfrak{sl}}}
\def\cA{{\mathcal{A}}}
\def\cB{{\mathcal{B}}}
\def\cC{{\mathcal{C}}}
\def\cD{{\mathcal{D}}}
\def\cE{{\mathcal{E}}}
\def\cF{{\mathcal{F}}}

\def\cI{{\mathcal{I}}}

\def\cK{{\mathcal{K}}}

\def\cN{{\mathcal{N}}}
\def\cO{{\mathcal{O}}}

\def\cQ{{\mathcal{Q}}}

\def\cU{{\mathcal{U}}}

\def\cW{{\mathcal{W}}}

\def\cZ{{\mathcal{Z}}}
\def\cHom{{\mathcal{H}om}}
\def\cExt{{\mathcal{E}xt}}

\def\PP{{\mathbb{P}}}

\def\CC{{\mathbb{C}}}
\def\id{{\operatorname{Id}}}
\def\sm{{\operatorname{sm}}}

\def\Ann{\operatorname{Ann}}

\def\codim{\operatorname{codim}}

\def\End{\operatorname{End}}
\def\Ext{\operatorname{Ext}}

\def\GL{\operatorname{GL}}
\def\Gr{\operatorname{Gr}}

\def\Hom{\operatorname{Hom}}

\def\Im{\operatorname{Im}}

\def\rank{\operatorname{rank}}

\def\SL{\operatorname{SL}}

\def\Sym{\operatorname{Sym}}

\def\Tot{\operatorname{Tot}}

\def\coKer{\operatorname{coKer}}
\def\coker{\operatorname{coker}}
\def\Ker{\operatorname{Ker}}
\def\depth{\operatorname{depth}}

\let\ordexists\exists
\def\exists{\operatorname{\ordexists}}
\let\ordforall\forall
\def\forall{\operatorname{\ordforall}}

\def\tilde{\widetilde}
\def\setminus{\smallsetminus}
\def\emptyset{\varnothing}

\def\longarrow#1#2{\mathchoice{#2}{#1}{#1}{#1}}
\def\to{\longarrow{\rightarrow}{\longrightarrow}}

\let\shortmapsto\mapsto
\def\mapsto{\longarrow{\shortmapsto}{\longmapsto}}

\usepackage[margin=1.2in]{geometry}

\title{On a conjecture on aCM and Ulrich sheaves on degeneracy loci}

\author[V.~Benedetti]{Vladimiro Benedetti}

\address{Universit\'e Côte d’Azur, CNRS - Laboratoire J.-A. Dieudonné, Parc Valrose, F-06108 Nice Cedex 2, France}
\email{vladimiro.benedetti@univ-cotedazur.fr}

\author[F.~Tanturri]{Fabio Tanturri}
\address{Dipartimento di Matematica, Dipartimento di Eccellenza 2023-2027, Università di Genova, Via Dodecaneso 35, 16146 Genova, Italy}\email{tanturri@dima.unige.it}

\begin{document}
\begin{abstract}
In this paper we address a conjecture by Kleppe and Mir\'o-Roig stating that suitable twists by line bundles (on the smooth locus) of the exterior powers of the normal sheaf of a standard determinantal locus are arithmetically Cohen--Macaulay, and even Ulrich when the locus is linear determinantal. We do so by providing a very simple locally free resolution of such sheaves obtained through the so-called Weyman's Geometric Method.
\end{abstract}

\maketitle

\section{Introduction}

Ulrich sheaves have been raising interest and attention in the past years since \cite{ESW}. In this paper it was conjectured that these sheaves, which are a special class of arithmetically Cohen--Macaulay (aCM) sheaves, exist on any projective variety. The existence of such sheaves on a variety has very interesting consequences both from an algebraic and a geometrical point of view. For instance, it was shown in \cite{ES} that it implies that the cone of possible cohomology tables of coherent sheaves on the variety is the same as the one for projective spaces.

From a more geometrical point of view, Ulrich sheaves are related to the existence of certain presentations of the equations of a projective variety. In \cite{beau1} it was shown that a hypersurface $\{F=0\}$ inside $\PP^n$ admits an Ulrich sheaf if and only if a power of the equation $F$ of the hypersurface can be written as the determinant of a suitable linear matrix; combined with the conjecture in \cite{ESW}, this says that any hypersurface is, modulo some power, \emph{determinantal}. We believe that this result may extend beyond the case of determinantal hypersurfaces to degeneracy loci, as hinted by \cref{prop_ulrich_determin}.

Since \cite{ESW}, much work has focused on constructing Ulrich sheaves on various classes of projective varieties. The case of hypersurfaces, and more generally of complete intersections, was dealt with in \cite{HUB}, where it was shown that smooth complete intersections admit Ulrich bundles. Then various authors have been constructing Ulrich sheaves on low dimensional varieties (see for instance \cite{ACMR,CFK}) or on homogeneous varieties (\cite{CMR,CCHMRW,CJ,LP}). In particular, for this last class, when restricting to completely reducible homogeneous bundles, one can even classify both aCM and Ulrich sheaves via combinatorial data (aCM sheaves on projective spaces were first classified in \cite{Ho} without any homogeneity assumption).

When looking for Ulrich sheaves, one may wonder whether well-known sheaves such as the tangent or cotangent bundle, and the normal or conormal bundle of a given variety are Ulrich. This question was essentially answered in a series of papers (\cite{BMPMT,Lo,LR,ACLR}) by classifying those varieties for which the mentioned bundles are Ulrich. In this classification a special role is played by degeneracy loci of morphisms of vector bundles.
\medskip

Given a morphism of vector bundles $s:E\to F$ on a variety $X$, one defines the $r$-th degeneracy locus of $s$ as the subvariety $D_r(s):=\{x\in X\mid \rank(s(x))\leq r\}$. These loci have been studied both in commutative algebra and in algebraic geometry. The main results in the context of Ulrich sheaves on degeneracy loci are contained in \cite{KMR}. In this technically advanced paper it was shown for instance that, when $X=\PP^n$, $r=\min(\rank(E),\rank(F))-1$ and $D_r(s)$ is of the expected dimension:
\begin{itemize}
\item[$i)$] if the matrix defining $s$ has polynomial entries, then a suitable twist of the normal bundle $\cN_{D_r(s)/ X}$ of $D_r(s)$ inside $X$ is aCM;
\item[$ii)$] if moreover the entries are linear, a suitable twist of $\cN_{D_r(s)/ X}$ is Ulrich and $\mu$-semistable.
\end{itemize}
In the above, the twist is meant to be ``twist by a line bundle on the smooth locus''. In the same paper, the authors conjectured that a similar behaviour should occur for the exterior powers of the normal bundle. The conjecture was motivated by the case of $\cN_{D_r(s)/ X}$ and by further numerical evidence. Our aim is to answer to that conjecture in a positive way (as a consequence of the results in \cref{sec_applications}, see also \cref{rem_almost_conj}).
\begin{theorem*}[Theorem A]
Let $X=\PP^n$, $r=\min(\rank(E),\rank(F))-1$ and $D_r(s)$ be of expected pure dimension. Then for any $j=0,\dots,\codim_X(D_r(s))$ there exist sheaves $\mathcal{W}_j$ which coincide on the smooth locus with $\wedge^j\cN_{D_r(s)/ X}$ up to a line bundle, such that
\begin{itemize}
\item[$i)$] if the matrix defining $s$ has polynomial entries, then $\mathcal{W}_j$ is aCM for $j=0,\dots,\codim_X(D_r(s))$;
\item[$ii)$] if moreover the entries are linear, a suitable twist of $\mathcal{W}_j$ is Ulrich and $\mu$-semistable for $j=0,\dots,\codim_X(D_r(s))$.
\end{itemize}
\end{theorem*}
We obtain this result in a \emph{uniform} way, by constructing a locally free minimal resolution of suitable $\mathcal{W}_j$. This is done by using the so-called Weyman's Geometric Method, which allows us to construct a locally free resolution on a very natural resolution of singularities of $D_r(s)$, known as Kempf collapsing; then, this locally free resolution is pushed forward to obtain a locally free resolution of the required sheaves $\mathcal{W}_j$ (see \cref{thm_res_ODL} and \cref{thm_res_determ_ext_normal}, which is the corresponding \emph{universal} affine version). More specifically, we prove the following

\begin{theorem*}[Theorem B]
Let $E_1$, $E_2$ be vector bundles of rank $n,m$ on a Cohen--Macaulay scheme $X$, $r=n-1$ and $s\in H^0(E_1^\vee \otimes E_2)$ a global section such that $D_r(s)$ has expected pure dimension (i.e.\ pure codimension $m-r$), then we obtain a locally free resolution $$ 0\to \cE^j_\bullet \to \cW_j:=\Sym^{m-r-j}\cK(s) \otimes \wedge^j \coker(s) \to 0 ,$$ 
where, for $0\leq u\leq m-r$,
$$
\begin{aligned}
\cE_{-u}^{j} = \bigoplus_{\lambda\in \cI_j,|\lambda|=u+r(m-r-j)-j}\tilde{c}_{j\lambda}^\mu S^{\mu^t}E_1 \otimes S^{\tilde{\lambda}}E_2^\vee.
\end{aligned}
$$
\end{theorem*}
Here by $\cK(s)$ we mean the cokernel of $s^\vee:F^\vee\to E^\vee$, and by $\coker(s)$ the cokernel of $s$, knowing that $n\leq m$. $S^\bullet$ denotes the plethysm associated to the weight $\bullet$; $\mu^t$ is the transpose of the weight $\mu$, $|\mu |$ is the norm of $\mu$, while the meaning of $\tilde{\lambda}$ is explained in \cref{sec_weyman_resolution}. The coefficients $\tilde{c}^\mu_{j\lambda}$ are defined in \cref{sec_weyman_resolution} as well as the set $I_j$. By \cref{prop_normal_Kcoker_relative} the sheaf $\Sym^{m-r-j}\cK(s) \otimes \wedge^j \coker(s)$, when restricted to the smooth locus, is a twist of the $j$-th exterior power of the normal bundle of $D_r(s)$ inside $X$.
\medskip

The plan of the paper is as follows. In \cref{sec_repr_homog_multilin_ulrich} we recall basic facts about Ulrich sheaves, representation theory, homogeneous geometry and multilinear algebra. In \cref{sec3} we recall Weyman's Geometric Method and apply it to obtain the universal affine version of Theorem B. Then, in \cref{sec_res_deg_locy}, we apply this result to degeneracy loci to obtain Theorem B. Finally, in \cref{sec_applications} we deduce some applications, namely the results summarised in Theorem A.


\subsection*{Acknowledgments} This paper issued from discussions started in the Będlewo conference \emph{Explicit Algebraic Geometry}. We thank Angelo Felice Lopez for giving an inspiring talk during the conference, and Jan O.\ Kleppe for very useful explanations about \cite{KMR}. The first author would like to thank Micha\l\ Kaputska for his hospitality during the week prior to the conference.

VB was partially supported by FanoHK ANR-20-CE40-0023. FT is a member of INdAM GNSAGA. He was partially supported by the MIUR Excellence Department Project awarded to the Dipartimento di Matematica of the Università di Genova, CUP D33C23001110001, and by the Curiosity Driven 2021 Project \emph{Varieties with trivial or negative canonical bundle and the birational geometry of moduli spaces of curves: a constructive approach} - Programma nazionale per la Ricerca (PNR) DM 737/2021. 

\section{A few preliminaries}
\label{sec_repr_homog_multilin_ulrich}

In this section we recall the definition of arithmetically Cohen--Macaulay and Ulrich sheaves. Then we will recall some classical results about representation theory, in particular formulas for exterior powers of tensor products and tensor product by exterior powers of a vector space, and a very powerful tool in the study of homogeneous varieties which is Bott--Borel--Weyl Theorem. We end the section with some general lemmas on multilinear algebra.

\subsection{ACM and Ulrich sheaves}

Let us begin with some preliminary definitions. Let $X \subset \PP^N$  be a projective scheme and let $E$ be a coherent sheaf on $X$. The sheaf $E$ is
said to be \emph{locally Cohen--Macaulay} if $\depth E_x=\dim \cO_{X,x}$ for all (closed) points $x\in X$. It is said to be \emph{arithmetically Cohen--Macaulay} (denoted by aCM in the following) if it is locally Cohen--Macaulay and satisfies
\[
H^i(X, E(t)) = 0 \mbox{ for all }t \mbox{ and }i = 1, \dots, \dim X-1.
\]

$E$ is said to be \emph{initialized} if $H^0(E)\neq 0$ but $H^0(E(-1))=0$. An \emph{Ulrich} sheaf is an initialized aCM sheaf having the maximum possible number of global sections, i.e., $h^0(E)=\deg(X)\rank(E)$. 

In what follows, we will often identify a sheaf $E$ on $X \subset \PP^N$ with $i_*E$, where $i:X\to \bP^n$ denotes the inclusion. Ulrich sheaves can be equivalently characterised in several ways:

\begin{theorem}[\cite{ESW}, Proposition 2.1]
\label{ESW21}
Let $E$ be a coherent, $n$-dimensional sheaf on $\PP^N$ supported on a positive-dimensional scheme $X$. The following conditions are equivalent:
\begin{itemize}
    \item[$(i)$] $E$ is Ulrich.
\item[$(ii)$] $E$ admits a linear $\cO_{\PP^N}$-resolution of the form
$$ 0 \to \cO_{\PP^N} (-N + n)^{a_{N-n}} \to \dots \to \cO_{\PP^N} (-1)^{a_1} \to O^{a_0}_{\PP^N} \to E \to 0.$$
\item[$(iii)$] $H^i(E(-i)) = 0$ for $i > 0$ and $H^i(E(-i - 1)) = 0$ for $i < n$.
\item[$(iv)$] For some (resp.\ all) finite linear projections $\pi : X \to \PP^n$, the sheaf $\pi_*E$ is the trivial sheaf $O^t_{\PP^n}$ for some $t$.
\end{itemize}
\end{theorem}

This simple result shows how degeneracy loci are intertwined with Ulrich sheaves.

\begin{proposition}
\label{prop_ulrich_determin}
Let $X\subset \bP^N$ be a codimension $c$, degree $d$ subvariety, and let $E$ be an Ulrich sheaf of rank $r$ on $X$. Then $X$ is cut out set-theoretically by the maximal minors of a $cdr\times dr$ matrix of linear forms.
\end{proposition}

\begin{proof}
Since $E$ is Ulrich, there exists a linear resolution of $E$ as in \cref{ESW21}. By \cite{ESW}, the ranks $a_i$ of the terms in the resolution are determined by the degree of $E$, which is $dr$, hence the resolution has the following form
$$ 0\to \cO_{\PP^N}(-c)^{dr} \to \cdots \to \cO_{\PP^N}(-1)^{cdr} \to \cO_{\PP^N}^{dr}\to E \to 0.$$
The sheaf $E$ is supported on $X$ and it is the cokernel of the morphism $\cO(-1)^{cdr} \to \cO^{dr}$ of the resolution. In particular, its set-theoretical support is the locus where this morphism has rank at most $dr-1$.
\end{proof}

In what follows, we will always assume that all schemes are projective.

\subsection{Technical tools from representation theory}

If $V$ is a vector space and $\lambda$ a non-increasing sequence of integers, we denote by $S^\lambda V$ the Schur functor applied to $V$. For instance $S^{(p)}=\Sym^p V$ and $S^{(1^{k})}V=\wedge^k V$. Since this is the only case we will need later on, we restrict ourselves to the following hypothesis: $\bm{r=n-1}$; also recall that $n \leq m$ by hypothesis. The three technical ingredients of the following sections are Bott--Borel--Weil Theorem and a decomposition of the exterior powers of a tensor product and of the tensor product of particular Schur functors. Before going on let us recall how these tools work.

\begin{remark}[Bott--Borel--Weil Theorem]
Later on we will need to compute the cohomology groups of homogeneous bundles over the Grassmannian $\Gr(r,m)$. We will use Bott--Borel--Weil Theorem to do so. We will adopt a notation \emph{ad hoc}: a homogeneous vector bundle $U$ on $\Gr(r,m)$ will be denoted by a sequence of integers $[\mu_m,\mu_{m-1},\cdots,\mu_{m-r+1};\mu_{m-r},\cdots,\mu_1]$ with $\mu_m\geq \cdots  \geq\mu_{m-r+1}$ and $\mu_{m-r}\geq \cdots \geq \mu_1$, where the sequence $\mu$ corresponds to the dual of the highest weight of the $\SL(r)\times \SL(m-r)\rtimes \CC^*$-representation defined by the fibres of $U$. For instance $\cU=[0,\cdots,0,-1;0,\cdots,0]$, $\cO(-1)=[a,\cdots,a;a+1,\cdots,a+1]$ for any integer $a$, $\cU^\vee=[1,0,\cdots,0]$, $\cQ=[0,\cdots,0;0,\cdots,-1]$, $\cQ^\vee=[0,\cdots,0;1,\cdots,0]$. In general the bundle $S^\mu \cU^\vee \otimes S^\nu  \cQ^\vee$ obtained by applying the Schur functors $S^\mu$ and $S^\nu$ corresponds to the weight $[\mu;\nu]$. 

Let us also denote by $\delta:=[m,m-1,\cdots,1]$. Then Bott--Borel--Weil Theorem states that a bundle $[\mu]$ has cohomology in degree $l$ if and only if $l$ is the length of the permutation $\sigma$ such that $\sigma(\mu+\delta)$ is a strictly decreasing sequence. If such a permutation exists (i.e.\ $\mu_i+\delta_i \neq \mu_j+\delta_j$ for any $i\neq j$) then $H^l(\Gr(r,m),[\mu]) =S^{\sigma(\mu+\delta)- \delta}(\CC^m)^\vee$.
\end{remark}

\begin{remark}[Exterior power of a tensor product]
Recall the decomposition of the exterior power of a tensor product. Indeed, if $A,B$ are two vector spaces then $\wedge^p (A\otimes B)$ decomposes in a direct sum of vector spaces as follows:
$$ \wedge^p (A\otimes B)=\bigoplus_{|\mu|=p} S^\mu A \otimes S^{\mu^t}B ,$$
where $|\mu|=\sum_i \mu_i$ and $\mu^t$ is the non-increasing sequence corresponding to a Young diagram which is dual to that of $\mu$. This is a consequence of the Littlewood--Richardson rule (see \cite[Exercise 6.11]{FultonHarris}). Notice that this formula when $A$ is one-dimensional yields the well-known formula $\wedge^p (A\otimes B) = A^{\otimes p}\otimes \wedge^p B$ because all other terms vanish. Notice more generally that in the formula above one can consider only weights $\mu$ such that $\mu_{\dim(A)+1}=0$ and $\mu_1\leq \dim(B)$ otherwise either $S^\mu A$ or $S^{\mu^t}B$ vanish. Clearly when $A$ and $B$ are representations of a reductive group, the formula above respects the decomposition in irreducible representations.
\end{remark}

\begin{remark}[Tensor product by exterior power]
Let $\lambda=[\lambda_{m-r}\geq \cdots \geq \lambda_1]$ be an integer sequence, $0\leq j\leq m-r$ and $A$ a $(m-r)$-dimensional vector space. Let us denote by $c_{j\lambda}^\mu$ the coefficients appearing in the decomposition
\begin{equation}
\label{tensor_exterior}
    S^\lambda A \otimes \wedge^j A = \bigoplus_{|\mu|=|\lambda|+j}c_{j\lambda}^\mu S^\mu A.
\end{equation}
It turns out from the Littlewood--Richardson rule that $c_{j\lambda}^\mu$ is either equal to $0$ or $1$. All weights $\mu$ for which $c_{j\lambda}^\mu=1$ are recovered as follows from $\lambda$. Denote by $Y_\lambda$ (respectively $Y_\mu$) the Young diagram of $\lambda$ (resp.\ $\mu$). Recall that $Y_\lambda$ is a diagram with $m-r$ columns and $\lambda_i$ boxes in the $i$-th column (column indexes decrease towards the right, so that $Y_\lambda$ has a staircase shape decreasing towards the right). Then $c_{j\lambda}^\mu=1$ if and only if $Y_\mu$ is obtained from $Y_\lambda$ by adding $j$ boxes so that
\begin{itemize}
    \item[$-$] the new diagram has a staircase shape decreasing towards the right;
    \item[$-$] no two added boxes are one above the other.
\end{itemize}
We will use this rule even with weights $\lambda$ such that $\lambda_1<0$. In this case apply the rule to $\lambda'=\lambda+[\lambda_1,\cdots,\lambda_1]$ in order to obtain coefficients $c_{j\lambda'}^{\mu'}$; then use \cref{tensor_exterior} by noticing that $c_{j\lambda}^{\mu}=c_{j\lambda'}^{\mu'}$ when $\mu=\mu'-[\lambda_1,\cdots,\lambda_1]$.
\end{remark}

\subsection{Some multilinear algebra}

Let us consider an exact sequence of modules over a ring $R$:
\[
M \mathrel{\mathop{\to}^{\mathrm{\varphi}}} N \to \coker\varphi \to 0.
\]
We can construct $N^{\otimes k}$ and its quotients $\wedge^k N$ and $\Sym^k N$. By the right exactness of the tensor product we have for each $k \geq 1$ an induced exact sequence
$$ M\otimes N^{\otimes k-1}\mathrel{\mathop{\to}^{\mathrm{\varphi_{k}^1}}}N^{\otimes k}\to \coker \varphi \otimes N^{\otimes k-1}\to 0 ;$$
we also obtain analogous exact sequences $\varphi_{k}^i$ for $1\leq i\leq k$ by exchanging the order of the factors of the tensor product. We have two natural quotient maps $\wedge:N^{\otimes k}\to \wedge^k N$ and $\bullet:N^{\otimes k}\to \Sym^k N$. We will denote by $\wedge^{k-1}N\wedge M$ the image of the composition $\wedge \circ(\sum_i \varphi_k^i)$; we will denote by $\Sym^{k-1}N\cdot M$ the image of the composition $\bullet \circ(\sum_i \varphi_k^i)$. 

\begin{lemma}
\label{lem_image_varphi1}
The image of $\wedge \circ \varphi_k^1$ is equal to $\wedge^{k-1}N\wedge M$. The image of $\bullet \circ \varphi_k^1$ is equal to $\Sym^{k-1}N\cdot M$.
\end{lemma}

\begin{proof}
For both statements, we just need to show that an element of the form $n_1 \otimes n_2 \otimes \dotsc \otimes n_{k}$ such that $n_j = \varphi(m)$ for some $j$ and some $m \in M$ represents the same (respectively skew-symmetric or symmetric) equivalence class of an element $\varphi(m') \otimes n'_2 \otimes \dotsc \otimes n'_{k}$ for some $m' \in M, n'_i \in N$. For the first statement we observe that
\begin{multline*}
(\varphi(m) + n_1) \otimes n_2 \otimes \dotsc \otimes n_{j-1} \otimes (\varphi(m) + n_1) \otimes n_{j+1} \otimes \dotsc \otimes n_k = \\
= \varphi(m) \otimes n_2 \otimes \dotsc \otimes n_{j-1} \otimes \varphi(m) \otimes n_{j+1} \otimes \dotsc \otimes n_k + \\
+ n_1 \otimes n_2 \otimes \dotsc \otimes n_{j-1} \otimes n_1 \otimes n_{j+1} \otimes \dotsc \otimes n_k + \\
+ \varphi(m) \otimes n_2 \otimes \dotsc \otimes n_{j-1} \otimes n_1 \otimes n_{j+1} \otimes \dotsc \otimes n_k + \\
+ n_1 \otimes n_2 \otimes \dotsc \otimes n_{j-1} \otimes \varphi(m) \otimes n_{j+1} \otimes \dotsc \otimes n_k;
\end{multline*}
since the first two summands and the LHS of the equality correspond to tensors with repetitions, we conclude by choosing $m' = -m, n_i = n'_i$ for any $2 \leq i \neq j$ and $n'_j = n_1$.

Since $(n_1\otimes n_2\otimes \cdots \otimes n_k) - (n_j \otimes n_2 \otimes \cdots \otimes n_{j-1} \otimes n_1 \otimes n_{j+1} \otimes \cdots \otimes n_k)$ represents the zero class inside the symmetric product, for the second statement we simply choose $m' = m, n_i = n'_i$ for any $2 \leq i \neq j$ and $n'_j = n_1$. 
\end{proof}

The following lemma is well-known for vector spaces and vector bundles. For the sake of completeness we provide here a proof for modules over a ring. 
\begin{lemma}
\label{lem_wedge_blink_blink}
Any short exact sequence of modules over a ring
\[
M \mathrel{\mathop{\to}^{\mathrm{\varphi}}} N \to \coker\varphi \to 0
\]
induces, for any $k \in \mathbb{N}, k \geq 2$  an exact sequence
\[
\wedge^{k-1}N \wedge M \to \wedge^k N \to \wedge^k \coker \varphi \to 0.
\]
\end{lemma}
\begin{proof}
Let $C:=\coker \varphi$. By definition the module $\wedge^k C$ is the quotient of $C^{\otimes k}$ by the submodule generated by tensors with repetitions. We have a naturally induced map $N^{\otimes k} \to C^{\otimes k}$, whose kernel is $\Im(\sum_i \varphi_k^i)$, see e.g.\ \cite[II, \textsection 3.6, Proposition 6]{Bourbaki}; composing it with the projection onto $\wedge^k C$ we get an alternating map which factorizes via $N^{\otimes k} \to \wedge^k N$, yielding a map $\wedge^k N \to \wedge^k C$. We have the following commutative diagram
\[
\xymatrix{0 \ar[r] & \Im(\sum_i \varphi_k^i) \ar[r] & N^{\otimes k} \ar[r]\ar[d]^-{\wedge} & C^{\otimes k} \ar[r]\ar[d] & 0\\
& & \wedge^k N\ar[d] \ar[r] & \wedge^k C \ar[d]\ar[r] & 0 \\
& & 0 & 0}
\]
and it straightforward to see that the kernel of $\wedge^k N \to \wedge^k C$ is $\wedge^{k-1}N\wedge M$.
\end{proof}

\begin{lemma}
\label{lem_sym_blink_blink}
Any short exact sequence of modules over a ring
\[
M \mathrel{\mathop{\to}^{\mathrm{\varphi}}} N \to \coker\varphi \to 0
\]
induces, for any $k \in \mathbb{N}, k \geq 2$ an exact sequence
\[
\Sym^{k-1}N \cdot M \to \Sym^k N \to \Sym^k \coker \varphi \to 0.
\]
\end{lemma}
\begin{proof}
We argue as in the proof of \cref{lem_wedge_blink_blink} by looking at the commutative diagram
\[
\begin{gathered}[b]
\xymatrix{0 \ar[r] & \Im(\sum_i \varphi_k^i) \ar[r] & N^{\otimes k} \ar[r]\ar[d]^-{\bullet} & C^{\otimes k} \ar[r]\ar[d] & 0\\
& & \Sym^k N\ar[d] \ar[r] & \Sym^k C \ar[d]\ar[r] & 0 \\
& & 0 & 0} \\[-\dp\strutbox]
\end{gathered}
\qedhere
\] 
\end{proof}

\section{Some resolutions on determinantal varieties of minimal corank}
\label{sec3}

The goal of this section is to obtain \cref{thm_res_determ_ext_normal}, which is the universal affine version of Theorem B. In order to do so, we recall the Geometric Method, as it is presented in \cite{Weyman}, and we use it to obtain locally free resolutions of certain sheaves on determinantal varieties. The description of these sheaves in terms of canonical sheaves on determinantal varieties will be given in \cref{prop_expr_Wj_cokerK}, and a special case is described in \cref{sec_j1}. This constitutes the technical heart of the paper.

\subsection{Kempf collapsings and the Geometric Method}

\label{sec_weyman_resolution}
Let $M_{m,n}$ be the space of $m\times n$ matrices seen as an affine variety, with $m\geq n$. For any $r \leq n$ the space $Y_r\subset M_{m,n}$ of matrices of rank at most $r$ is a codimension $(m-r)(n-r)$ affine subvariety of $M_{m,n}$, singular along $Y_{r-1}$, usually referred to as \emph{determinantal variety}. Notice that this is not the same as (standard or linear) determinantal loci, that we will define in \cref{sec_applications}.

The variety $Y_r$ admits a resolution of singularities provided by the total space of a vector bundle over an algebraic variety: this is an instance of a particular situation known as \emph{Kempf collapsing} (see \cite{Kem}). Thinking of the elements of $M_{m,n}$ as maps $(\CC^n)^\vee \otimes \CC^m$, on the Grassmannian $\Gr(r,\CC^m)$ we consider the vector bundle $(\CC^n)^\vee \otimes \cU = \cU^n$. Its total space $\Tot(\cU^n)$ can be thought of as a subvariety of $\Tot(\cO_{\Gr(r,m)}^{mn})\cong M_{m,n}\times \Gr(r,m)$. We are in the following situation:

\begin{equation}
\label{kempfcoll}
\xymatrix{
\Tot(\cU^n)\rule{2pt}{0pt} \ar@{^{(}->}[dr] \ar@/^/@<1ex>[rrd]^-{{p}} \ar@/_/@<-1ex>[drd]_-{{q}}\\
& M_{m,n}\times \Gr(r,m) \ar[r]_-{\tilde{p}} \ar[d]^-{\tilde{q}} & \Gr(r,m) \\
& M_{m,n}
}
\end{equation}

A point $V_m \subset \CC^m$ should be thought of as an $m$-dimensional space containing the image of a matrix, and a point in $\Tot(\cU^n)$ over $V_m \subset \CC^m$ should be thought of as the choice of the images of a basis in $\CC^n$.

Let us denote by $\xi:=\cO_{\Gr(r,m)}^{mn}/\cU^n=\cQ^n$ and by $e:=\rank(\xi)=n(m-r)$. Following \cite[Chapter 5]{Weyman}, there exists a \emph{Koszul complex} of the form
\begin{multline}
\label{koszulComplex}
0 \to \wedge^{e}\xi^\vee \otimes \cO_{M_{m,n}\times \Gr(r,m)}\to \wedge^{e-1}\xi^\vee \otimes \cO_{M_{m,n}\times \Gr(r,m)}\to \cdots \to \\
\to \xi^\vee \otimes \cO_{M_{m,n}\times \Gr(r,m)} \to \cO_{M_{m,n}\times \Gr(r,m)} \to 0,
\end{multline}
exact everywhere but in the last step, providing an $\cO_{M_{m,n}\times \Gr(r,m)}$-locally free resolution of $\cO_{\Tot(\cU^n)}$. 

Weyman's Geometric Method (\cite{Weyman}) consists in computing various $\cO_{M_{m,n}}$-locally free resolutions of sheaves on $Y_r$ as pushforwards via $\tilde{q}$ of the Koszul complex \cref{koszulComplex} twisted by the pullback of suitable sheaves on $\Gr(r,m)$. Let us be more precise. We will be interested in the coherent sheaves $\tilde{q}_*({\tilde{p}}^*(W_j) \otimes \cO_{\Tot(\cU^n)})$ on $Y_r$ for
\begin{equation*}
W_j:=\det(\cU)^{m-r-j}\otimes \wedge^j\cQ, \qquad \quad 0\leq j\leq m-r.
\end{equation*}

By computing the pushforward $\tilde{q}_*$ of the Koszul complex \cref{koszulComplex} twisted by ${\tilde{p}}^*(W_j)$ we will be able to construct a complex $\cF^{W_j}_\bullet: \cdots \to \cF^{W_j}_{-1}\to \cF^{W_j}_0 \to \cF^{W_j}_{1}\to \cdots$ whose terms are of the form:
$$ \cF^{W_j}_{u}:=\bigoplus_{l-p=u} H^l(\Gr(r,m),\wedge^p \xi^\vee \otimes {W_j})\otimes \cO_{M_{m,n}}(-p) ,$$
where $\cO_{M_{m,n}}(-p)$ is a trivial twist on the affine space $M_{m,n}$ which corresponds to the $p$-th twisted tautological bundle on $\bP(M_{m,n})$. The meaning of the twist is essentially the following: a morphism $\phi:W\otimes \cO_{M_{m,n}}(-i) \to V \otimes \cO_{M_{m,n}}(j)$, where $V,W$ are vector spaces, is given by an element $$\phi\in W^\vee \otimes V \otimes \Sym^{j+i}(M_{m,n}^\vee)\subset W^\vee \otimes V \otimes \Sym(M_{m,n}^\vee)=W^\vee \otimes V \otimes \Gamma_*(M_{m,n},\cO_{M_{m,n}}).$$

\subsection{Pushforward of \texorpdfstring{$W_j$}{W\_j}}

In this section we want to study the pushforward to $Y_r$ of the vector bundle $W_j:= \det(\cU)^{m-r-j}\otimes \wedge^j \cQ$ for $0\leq j \leq m-r$ and $r=n-1$. Let us denote by
$$ \cI_j:=\{ \lambda=[\lambda_{m-r}\geq \cdots \geq \lambda_1]\mid n\geq \lambda_{m-r}, \lambda_{j+1}\geq r,0\geq \lambda_j,\lambda_1\geq -1\}. $$
Let us also define
$$
\tilde{c}_{j\lambda}^\mu:=\begin{cases}
0 \mbox{ if }\mu_1<0\\
c_{j\lambda}^\mu \mbox{ if }\mu_1\geq 0
\end{cases}
$$
and
$$
\tilde{\lambda}:=[\lambda_{m-r}-r,\cdots,\lambda_{j+1}-r,0^r,\lambda_j,\cdots,\lambda_1],
$$
where the coefficients $c^\mu_{j\lambda}$ are defined in \cref{tensor_exterior}.

\begin{proposition}
\label{prop_cohom_gp_ext}
Let $n\leq m$, $r=n-1$ and $c=m-r=\codim_{M_{m,n}}Y_r$; then $\cF^{W_j}_u=0$ for $u>0$ and $u<-m+r$. For $0\leq u \leq m-r$ we have the following isomorphisms:
$$
\begin{aligned}
\cF_{-u}^{W_j} = \bigoplus_{\lambda\in \cI_j,|\lambda|=u+r(m-r-j)-j}\tilde{c}_{j\lambda}^\mu S^{\mu^t}\CC^n \otimes S^{\tilde{\lambda}}(\CC^m)^\vee \otimes \cO(-u-r(m-r-j)).
\end{aligned}
$$
\end{proposition}

\begin{proof}
We need to compute the cohomology groups of the bundle $$\wedge^p \xi^\vee \otimes W_j= \wedge^p (\CC^n \otimes \cQ^\vee) \otimes \wedge^j \cQ(-m+r+j)$$ over the Grassmannian $\Gr(r,m)$. By applying the formula for the exterior power of a tensor product to $\wedge^p \xi^\vee$ we obtain that $\wedge^p (\CC^n\otimes \cQ^\vee)=\wedge^p (\CC^n\otimes [0,\cdots,0;1,\cdots,0])$ is given by a direct sum of terms of the form $S^{\mu^t} \CC^n \otimes [0,\cdots,0;\mu_{m-r},\cdots,\mu_1]$ with $\sum_i \mu_i=p$, $n\geq \mu_{m-r}\geq \cdots \geq \mu_1\geq 0$.
When we tensor each of these terms by $\wedge^j \cQ(-m+r+j)$ we obtain a direct sum of terms of the form $S^{\mu^t}\CC^n\otimes [-m+r+j,\cdots,-m+r+j;\lambda_{m-r},\cdots,\lambda_1]$ where $\lambda$ is a decreasing sequence of integers with each integer $\lambda_i$ equal to $\mu_i$ except for $j$ indexes $\iota_1,\cdots,\iota_j$ for which $\lambda_{\iota_1}=\mu_{\iota_1}-1,\cdots,\lambda_{\iota_j}=\mu_{\iota_j}-1$. Notice that $0\leq p\leq n(m-r)$. 

In order to compute the cohomology, let us add $\delta$ to our sequence to obtain $[r+j,r+j-1,\cdots,j+1;\lambda_{m-r}+m-r,\cdots,\lambda_1+1]$. The terms which will contribute to the cohomology are those such that $\{\lambda_{m-r}+m-r,\cdots,\lambda_1+1\} \cap [j+1,r+j]=\emptyset$. Thus $\lambda_{j+1}\geq n-1$ and either $\lambda_{j}\geq n$ or $\lambda_j\leq 0$. The first case is not possible because otherwise we would obtain $\lambda_{m-r}=\cdots=\lambda_j=n$ and $\lambda_{j-1},\cdots,\lambda_1\geq -1$; this implies that $\mu_{m-r}=n+1$ which is not allowed since then $S^{\mu^t}\CC^n=0$. So we deduce that $0\geq\lambda_j\geq \cdots\geq \lambda_1\geq -1$. In all cases the permutation that transforms $[r+j,r+j-1,\cdots,j+1;\lambda_{m-r}+m-r,\cdots,\lambda_1+1]$ into a strictly decreasing sequence has length $r(m-r-j)$, so the cohomology is concentrated in degree $r(m-r-j)$. Let us find all terms which give a non-trivial cohomology group.
\medskip

In order to conclude, first notice that $\lambda$ belongs to $\cI_j$ (the fact that $\lambda_{m-r}\leq n$ is once more a consequence of the fact that otherwise $S^{\mu^t}\CC^n=0$). Then we want to recover all $\mu$'s such that $S^\lambda \cQ^\vee \subset S^\mu \cQ^\vee \otimes \wedge^j \cQ$ with the extra condition that $\mu_1\geq 0$. However this is the same as asking both that $S^\mu \cQ\subset S^\lambda \cQ \otimes \wedge^j \cQ$ and $\mu_1\geq 0$. These two conditions are encoded in the fact that $\tilde{c}_{j\lambda}^\mu\neq 0$, which is the same as $\tilde{c}_{j\lambda}^\mu=1$. Finally one easily sees that the cohomology of $[-m+r+j,\cdots,-m+r+j;\lambda_{m-r},\cdots,\lambda_1]$ over the Grassmannian is equal to $S^{(\lambda_{m-r}-r,\cdots,\lambda_{j+1}-r,0^r,\lambda_j,\cdots,\lambda_1)}(\CC^m)^\vee$ by applying Bott--Borel--Weil Theorem.
\end{proof}

\begin{theorem}
\label{thm_res_determ_ext_normal}
Let $r=n-1$; then the complex $$0\to \cF^{W_j}_\bullet \to q_*p^* W_j \to 0$$ is exact and $\GL(m)\times \GL(n)$-equivariant. 
\end{theorem}

\begin{proof}
The projection $q:\Tot(\cU^n) \to Y_r$ is birational. Then we can apply \cite[Theorem 5.1.2]{Weyman}. Since $\cF^{W_j}_u=0$ for $u>0$ by \cref{prop_cohom_gp_ext} then the complex $\cF^{W_j}_\bullet$ is left exact and resolves $q_*p^*W_j$. Since all constructions are $\GL(m)\times \GL(n)$-equivariant, the same holds for $\cF^{W_j}_\bullet$ by \cite[Theorem 5.4.1]{Weyman}.
\end{proof}

In particular, the first two terms of the complex $\cF^{W_j}_\bullet$ are given by:
$$ \cF^{W_j}_0= \cO_{M_{m,n}}(-r(m-r-j))\otimes \det(\CC^n)^{m-r-j} \otimes \Hom(\Sym^{m-r-j}\CC^n,\wedge^j \CC^m);$$

\begin{equation}
\label{eq_FWjminus1_expression}
\begin{split}
    \cF^{W_j}_{-1} = {} & \cO_{M_{m,n}}(-r(m-r-j)-1)\otimes\det(\CC^n)^{m-r-j}\otimes \\ & \Big[ \Hom(\Sym^{m-r-j-1}\CC^n, \wedge^{j-1}\CC^m)  \oplus \\ &
    \oplus \Hom(S^{(m-r-j,0^{n-2},-1)}\CC^n, \wedge^{j-1}\CC^m) \oplus \\ &
    \oplus \Hom(\Sym^{m-r-j-1}\CC^n,S^{(1^j,0^{m-j-1},-1)}\CC^m)\Big].
    \end{split}
    \end{equation}

Notice that
\begin{equation}
    \label{eq_Aj_dec}
    \begin{split}
\Hom(\Sym^{m-r-j}\CC^n,\CC^n)= {} & \Sym^{m-r-j}(\CC^n)^\vee\otimes \CC^n \\ {} = {} & \Sym^{m-r-j-1}(\CC^n)^\vee \oplus S^{(m-r-j,0^{n-2},-1)}(\CC^n)^\vee 
\end{split}
\end{equation}
and
\begin{equation}
    \label{eq_Bj_dec}
\Hom(\CC^m,\wedge^{j}\CC^m) = (\CC^m)^\vee \otimes \wedge^{j}\CC^m = \wedge^{j-1}\CC^m \oplus S^{(1^j,0^{m-j-1},-1)}\CC^m.
\end{equation}

Let us denote by $d_{1}:\cF^{W_j}_{-1} \to\cF^{W_j}_0$. In what follows we will repeatedly use the fact that two sheaves are isomorphic if there exists a morphism which induces an isomorphism at each stalk. We want to precisely identify what $q_*({\tilde{p}}^*(W_j) \otimes \cO_{\Tot(\cU^n)})$ is. We will follow three steps, summarised in the following
\begin{strategy}
\label{threesteps}
\begin{itemize}
    \item[i)] show that $\cF^{W_j}_{-1}=L\otimes [ (\cA_j \oplus \cB_j) / \cC_j]$ for three (locally free) $\GL(m)\times \GL(n)$-equivariant sheaves $\cA_j, \cB_j,\cC_j$, and a line bundle $L$ (\cref{firststep});
    \item[ii)] show that there exists a $\GL(m)\times \GL(n)$-invariant morphism $\tilde{d_1}^j: \cA_j\oplus \cB_j \to \cF^{W_j}_0 \otimes L^{-1}$, of which we understand what the cokernel is (\cref{lem_ex_direct_sum});
    \item[iii)] show that the kernel of $\tilde{d_1}^j$ is $\cC_j$ and the induced morphism on the quotient is $d_1^j$ (\cref{lem_tilde_d1}).
\end{itemize}
\end{strategy}

From these we will obtain an explicit description of $q_*({\tilde{p}}^*(W_j) \otimes \cO_{\Tot(\cU^n)})$ in terms of particular sheaves on the affine space $M_{m,n}\cong \Hom(\CC^n,\CC^m)$, which we now introduce. Recall that, by definition of affine variety, $H^0(M_{m,n},\cO_{M_{m,n}})=\Sym(\Hom(\CC^n,\CC^m)^\vee)$. This space of sections has a natural grading which, as mentioned above, allows us to define trivial twists $\cO_{M_{m,n}}(i)$ of the structure sheaf $\cO_{M_{m,n}}$; the space of sections of such twists is $$H^0(M_{m,n},\cO_{M_{m,n}}(i))=\Sym^i(\Hom(\CC^n,\CC^m)^\vee).$$
Let us define the tautological morphism $t_\id:\CC^n\otimes \cO_{M_{m,n}}(-1) \to \CC^m \otimes \cO_{M_{m,n}}$ as the morphism associated to
\[
\id_{Hom(\CC^n,\CC^m)} \in \End(Hom(\CC^n,\CC^m)) = \Sym^1(Hom(\CC^n,\CC^m)^\vee) \otimes Hom(\CC^n,\CC^m).
\]
The restriction $(t_\id)_f$ of $t_\id$ over a fibre $f\in M_{m,n}$ is tautologically given by $f:\CC^n \to \CC^m$. We will denote by $\coKer$ the cokernel sheaf of $t_\id$. Notice that the kernel of $t_\id$ is a torsion-free sheaf which is generically zero, so it vanishes; thus we have an exact sequence of $\cO_{M_{m,n}}$-sheaves:
\begin{equation}
    \label{eq_def_coker}
0\to \CC^n \otimes \cO_{M_{m,n}} \to \CC^m \otimes \cO_{M_{m,n}} \to \coKer\to 0.
\end{equation}
Similarly, from the dual morphism $t_\id^\vee$ we can define a sheaf $K$ which satisfies the following exact sequence:
\begin{equation}
    \label{eq_def_K}
    0\to \coKer^\vee\to {\CC^m}^\vee \otimes \cO_{M_{m,n}} \to {\CC^n}^\vee \otimes \cO_{M_{m,n}} \to K \to 0,
\end{equation}
where $\coKer^\vee=\cHom_{\cO_{M_{m,n}}}(\coKer,\cO_{M_{m,n}})$ and $K=\cExt^1_{\cO_{M_{m,n}}}(\coKer,\cO_{M_{m,n}})$.

Having introduced $\coKer$ and $K$, we are ready to go on with \cref{threesteps}. 

\begin{lemma}
\label{firststep}
$\cF^{W_j}_{-1}$ admits the following $\GL(m)\times \GL(n)$-equivariant description: $$\cF^{W_j}_{-1}=\cO_{M_{m,n}}(-r(m-r-j))\otimes \det(\CC^n)^{m-r-j}\otimes (\cA_j\oplus \cB_j/\cC_j),$$
where
$$\cA_j:=\Sym^{m-r-j}(\CC^n)^\vee \otimes \CC^n \otimes \wedge^{j-1}\CC^m(-1) ,$$
$$ \cB_j:=\Sym^{m-r-j-1}(\CC^n)^\vee \otimes (\CC^m)^\vee\otimes \wedge^{j}\CC^m (-1) ,$$
$$ \cC_j:=\Sym^{m-r-j-1}(\CC^n)^\vee \otimes \wedge^{j-1}\CC^m(-1) .$$
\end{lemma}

\begin{proof}
It is sufficient to decompose $\cA_j$ and $\cB_j$ in $\GL(m)\times \GL(n)$-representations as it is done in \cref{eq_Aj_dec,eq_Bj_dec}, and compare $(\cA_j\oplus \cB_j)/\cC_j$ with \cref{eq_FWjminus1_expression}. Notice that $\cC_j$ appears once in both the decomposition of $\cA_j$ and $\cB_j$.
\end{proof}

Let us denote by $$\cD_j:=\cF^{W_j}_{0}\otimes \cO_{M_{m,n}}(r(m-r-j))\otimes \det(\CC^n)^{r+j-m}=\Sym^{m-r-j}(\CC^n)^\vee \otimes \wedge^j\CC^m \otimes \cO_{M_{m,n}}.$$ The tautological map $t_\id$ and the skew-symmetrising morphism $\wedge^{j-1}\CC^m \otimes \CC^m \to \wedge^{j}\CC^m$ induce a natural morphism $\tilde{d_1}^j_\cA:\cA_j \to \cD_j $; similarly, the dual tautological map $t_\id^\vee$ and the symmetrising morphism $\Sym^{m-r-j-1}(\CC^n)^\vee \otimes (\CC^n)^\vee \to \Sym^{m-r-j}(\CC^n)^\vee$ induce a natural morphism $\tilde{d_1}^j_\cB:\cB_j \to \cD_j $. Let us define $\tilde{d_1}^j:=\tilde{d_1}^j_\cA \oplus \tilde{d_1}^j_\cB$.

\begin{lemma}
\label{lem_ex_direct_sum}
There exists an exact sequence
$$ \cA_j\oplus \cB_j \mathrel{\mathop{\to}^{\mathrm{\tilde{d_1}^j}}} \cD_j \to \Sym^{m-r-j} K \otimes \wedge^j \coKer \to 0. $$
\end{lemma}

\begin{proof}
Let us decompose the morphism $$\tilde{d_1}^j_\cA:\Sym^{m-r-j}(\CC^n)^\vee \otimes \CC^n\otimes \wedge^{j-1}\CC^m(-1)\to \Sym^{m-r-j}(\CC^n)^\vee \otimes \wedge^j\CC^m.$$ This is just the composition of the tautological morphism $t_\id$ tensored by the map $\id_{\Sym^{m-r-j}(\CC^n)^\vee \otimes \wedge^{j-1}\CC^m}$, followed by the natural projection induced by $\wedge^{j-1}\CC^m\otimes \CC^m \to \wedge^j \CC^m$:
$$ \cA_j \mathrel{\mathop{\to}^{\mathrm{t_\id\otimes \id}}} \Sym^{m-r-j}(\CC^n)^\vee \otimes \wedge^{j-1}\CC^m \otimes \CC^m \to \cD_j . $$
Since $\Sym^{m-r-j}(\CC^n)^\vee \otimes \wedge^{j-1}\CC^m \cO$ is locally free, by tensoring \cref{eq_def_coker} with this sheaf we have an exact sequence 
$$
\cA_j \to \Sym^{m-r-j}(\CC^n)^\vee \otimes \wedge^{j-1}\CC^m \otimes \CC^m \to \Sym^{m-r-j}(\CC^n)^\vee \otimes \wedge^{j-1}\CC^m\otimes  \coKer;
$$
equivalently, $\tilde{d_1}^j_\cA(\cA_j)=\Sym^{m-r-j}(\CC^n)^\vee \otimes \wedge^{j-1}\CC^m \otimes t_\id(\CC^n)$. Then, by applying both \cref{lem_image_varphi1,lem_wedge_blink_blink} we obtain an exact sequence
\begin{equation}
\label{ex_seq_A1}
\cA_j \to \cD_j \to \Sym^{m-r-j}(\CC^n)^\vee \otimes \wedge^{j} \coKer\to 0,
\end{equation}
which in turn is equivalent to the fact that $\tilde{d_1}^j_\cA(\cA_j)=\Sym^{m-n-j}(\CC^n)^\vee \otimes (\wedge^{j-1}\CC^m \wedge t_\id(\CC^n)) $. 

Mutatis mutandis, a similar argument applies for $\tilde{d_1}^j_\cB$, this time by using $t_\id^\vee:(\CC^m)^\vee \to (\CC^n)^\vee$ and by applying \cref{lem_image_varphi1,lem_sym_blink_blink}; we obtain thus an exact sequence 
\begin{equation}
\label{ex_seq_B1}
\cB_j \to \cD_j \to \Sym^{m-r-j} K \otimes \wedge^j\CC^m,
\end{equation}
which in turn is equivalent to $\tilde{d_1}^j_\cB(\cB_j)=(\Sym^{m-r-j-1}(\CC^n)^\vee \cdot t_\id^\vee(\CC^m)^\vee) \otimes \wedge^j\CC^m$.

By \cite[II, \textsection 3.6, Proposition 6]{Bourbaki}, the tensor product of the two exact sequences \cref{ex_seq_A1,ex_seq_B1} induces the exact sequence stated in the lemma.
\end{proof}

\begin{lemma}
\label{lem_tilde_d1}
The kernel of $\tilde{d_1}^j$ is $\cC_j$; moreover, $d_1^j=(\tilde{d_1}^j \mod \cC_j )$.
\end{lemma}

\begin{proof}
First notice that $\cA_j\oplus \cB_j/\cC_j$ is the direct sum of three $\GL(m)\times\GL(n)$-equivariant sheaves which are the tensor product of $\cO_{M_{m,n}}(-1)$ with three $\GL(m)\times\GL(n)$-representations $E_{1},E_2,E_{3}$, where
$$E_1=S^{(m-r-j,0^{n-2},-1)}(\CC^n)^\vee \otimes \wedge^{j-1}\CC^m ,$$
$$E_2= \Sym^{m-r-j-1}(\CC^n)^\vee\otimes \wedge^{j-1}\CC^m,$$
$$E_3=\Sym^{m-r-j-1}(\CC^n)^\vee\otimes S^{(1^j,0^{m-j-1},-1)}\CC^m .$$
Moreover, for each $i=1,2,3$, there exists a unique non-zero $\GL(m)\times\GL(n)$-equivariant morphism $E_i\otimes M_{m,n} \to D_j:=\Sym^{m-r-j}(\CC^n)^\vee \otimes \wedge^j\CC^m $. The restriction of $d_1^j$ to each factor $E_i\otimes \cO_{M_{m,n}}(-1)$ is non-zero since otherwise $E_i\otimes \cO_{M_{m,n}}(-1) \subset \ker(d_1^j)$ and it would appear as a factor of $\cF^{W_j}_{-2}\otimes \cO_{M_{m,n}}(r(m-r-j))\otimes \det(\CC^n)^{r+j-m}$; however, this sheaf is a direct sum of factors all isomorphic to $\cO_{M_{m,n}}(-2)$, so this is not possible. Thus $d_1^j$ is the unique non-zero $\GL(m)\times\GL(n)$-equivariant morphism $\cF^{W_j}_{-1}\to \cF^{W_j}_0$.

Now notice that
\[
\cC_j\cong E_2\otimes \cO_{M_{m,n}}(-1), \qquad \cA_j=(E_1\oplus E_2)\otimes \cO_{M_{m,n}}(-1)
\]
and
\[
\cB_j=(E_2\oplus E_3)\otimes \cO_{M_{m,n}}(-1).
\]
As it turns out, $\tilde{d_1}^j$ is by construction different from zero when restricted to each of these factors. However, from the uniqueness of the non-zero $\GL(m)\times\GL(n)$-equivariant morphisms $E_2 \otimes M_{m,n} \to \Sym^{m-r-j}(\CC^n)^\vee \otimes \wedge^j\CC^m$ we deduce that the restriction of $\tilde{d_1}^j$ to the two $E_2\otimes \cO_{M_{m,n}}(-1)$ terms appearing in $\cA_j\oplus \cB_j$ is, modulo scalars, the same. Thus a diagonal copy of $E_2\otimes \cO_{M_{m,n}}(-1)=:\cC_j$ is contained and in fact is equal to $\ker(\tilde{d_1}^j)$. Then $(\tilde{d_1}^j \mod \cC_j)$ is a non-zero $\GL(m)\times\GL(n)$-equivariant morphism, and by the uniqueness of $d_1^j$ the result follows.
\end{proof}

\begin{proposition}
\label{prop_expr_Wj_cokerK}
We have the following identification
$$ q_*p^* W_j= \det(\CC^n)^{m-r-j} \otimes \Sym^{m-r-j}K\otimes \wedge^j\coKer\otimes \cO_{M_{m,n}}(-r(m-r-j)).
$$
\end{proposition}

\begin{proof}
This is a consequence of \cref{thm_res_determ_ext_normal,lem_ex_direct_sum,lem_tilde_d1}.
\end{proof}

We will end this section with a technical result which will be needed later on. Let us denote by $\cI_{Y_r}$ the ideal of $Y_r\subset M_{m,n}$. The ideal $\cI_{Y_r}$ is generated by $\{\det_{IJ} \}_{I,J}$, where $I=\{1\leq i_1<\cdots <i_{r+1}\leq m\}$, $J=\{1\leq j_1<\cdots <j_{r+1}\leq n\}$ and $\det_{IJ}$ is the determinant of the $(r+1)\times (r+1)$-minor defined by the indices $I$ and $J$.

\begin{proposition}
\label{prop_ann}
$\Ann_{\cO_{M_{m,n}}}q_*p^* W_j = \cI_{Y_r}$.
\end{proposition}

\begin{proof}
The module $q_*p^*W_j$ is a $q_*\cO_{\Tot(\cU^n)}=\cO_{Y_r}$-module, thus $\cI_{Y_r}\subset \Ann_{\cO_{M_{m,n}}}q_*p^* W_j$. Notice that $\Ann_{\cO_{M_{m,n}}}q_*p^* W_j=\Ann_{\cO_{M_{m,n}}}\Hom(\Sym^{m-r-j}\Ker,\wedge^j \coKer)$. Moreover the sheaf $\Hom(\Sym^{m-r-j}\Ker,\wedge^j \coKer)$ is locally free of constant non-zero rank on $Y_r\setminus Y_{r-1}$. Thus its support is $Y_r$. Since $\cO_{Y_r}$ is reduced, any equation $g\notin\cI_{Y_r}$ cannot be contained in $\Ann_{\cO_{M_{m,n}}}\Hom(\Sym^{m-r-j}\Ker,\wedge^j \coKer)$ because otherwise the support of $\Hom(\Sym^{m-r-j}\Ker,\wedge^j \coKer)$ would be contained in $Y_r\cap \{g=0\}$, a strict closed subvariety of $Y_r$. Thus we also deduce that $\Ann_{\cO_{M_{m,n}}}\Hom(\Sym^{m-r-j}\Ker,\wedge^j \coKer)\subset \cI_{Y_r}$.
\end{proof}

\subsection{The special case \texorpdfstring{$j=1$}{j=1}}
\label{sec_j1}

In this section 
we want to make the results in \cref{prop_cohom_gp_ext} and 
\cref{thm_res_determ_ext_normal} more explicit, at least when $j=1$.
In order to do so we rewrite \cref{prop_cohom_gp_ext} in this special case, as well as $\cF^{W_1}_{-1}$ and $\cF^{W_1}_0$, 
and then we explicitly compute the restriction $(d_{1})|_f$ of $d_{1}$ to a fibre $f\in M_{m,n}$. We start by collecting all results about $q_*p^*W_1$ when $j=1$.

\begin{proposition}
\label{prop_cohom_gp}
Let $n\leq m$, $r=n-1$ and let $c=m-r=\codim_{M_{m,n}}Y_r$; then $\cF^{W_1}_u=0$ for $u>0$ and $u<-m+r$. For $0\leq u \leq m-r$ we have the following isomorphisms:
$$
\begin{aligned}
\cF_{-u}^{W_1} \!= \! \left[ \left( S^{((c-1)^{n-1},u)}\CC^n \otimes  S^{(1^{u},0^{m-u-1},-1)}(\CC^m)^\vee \right) \oplus \left( S^{((c-1)^{n-1},u)}\CC^n \otimes  S^{(1^{u-1})}(\CC^m)^\vee \right) \oplus \right. \\ \left. \oplus \left( S^{(c,(c-1)^{n-2},u-1)}\CC^n \otimes  S^{(1^{u-1})}(\CC^m)^\vee \right)  \right] \otimes  \cO_{M_{m,n}}(-rc+r-u) .
\end{aligned}
$$
The complex $
0\to \cF^{W_1}_{\bullet} \to q_*p^*W_1 \to 0
$ is exact and $\GL(m)\times \GL(n)$-equivariant and
$$ q_*p^* W_1= \det(\CC^n)^{m-n} \otimes \Sym^{m-n}K\otimes \coKer\otimes \cO_{M_{m,n}}(-r(m-n)).
$$
\end{proposition} 

In particular, the first two terms of the complex $\cF^{W_1}_\bullet$ are given by:
$$ \cF^{W_1}_0= \cO_{M_{m,n}}(-r(m-n))\otimes \det(\CC^n)^{m-n}\otimes \Hom(\Sym^{m-n}\CC^n,\CC^m) ;$$
\begin{equation}
    \label{eq_F_11_expression}
\begin{aligned}
    \cF^{W_1}_{-1}=\cO_{M_{m,n}}(-r(m-n)-1)\otimes\left[\det(\CC^n)^{m-n}\otimes \Hom(\Sym^{m-n-1}\CC^n,\CC)\oplus\right. \\
    \oplus \det(\CC^n)^{m-n}\otimes \Hom(\Sym^{m-n-1}\CC^n,\fsl(\CC^m)) \oplus \\  \left. \oplus \det(\CC^n)^{m-n}\otimes \Hom(S^{(m-n,0^{n-2},-1)}\CC^n,\CC)\right] .
\end{aligned}
\end{equation}
Using \cref{eq_Aj_dec} we can rewrite the term $\cF^{W_1}_{-1}$ above as:
\[
\begin{aligned}
    \cF^{W_1}_{-1}=\cO_{M_{m,n}}(-r(m-n)-1)\otimes\det(\CC^n)^{m-n}\otimes \left[ \Hom(\Sym^{m-n}\CC^n,\CC^n)\oplus \right. \\ \left.
    \oplus \Hom(\Sym^{m-n-1}\CC^n,\fsl(\CC^m))  \right] .
\end{aligned}
\]
Let $b:\Sym^{m-n-1}\to \fsl(\CC^m)$ and $f:\CC^n\to \CC^m$. We can write $ b=\sum_i p_i \otimes g_i$ for $p_i\in \Sym^{m-n-1}(\CC^n)^\vee$ and $g_i\in \fsl(\CC^m)$, and $f=\sum_j x_j\otimes v_j$ for $x_j\in (\CC^n)^\vee$ and $v_j\in \CC^m$. Let us define \begin{equation}\label{def_circledcirc}
    b\circledcirc f:=\sum_{i,j} p_i\cdot x_j \otimes g_i(v_j) \in \Sym^{m-n}(\CC^n)^\vee\otimes \CC^m.
\end{equation}
    Let us denote by $d_1^1:\cF^{W_1}_{-1}\to \cF^{W_1}_{0}$. For $f\in M_{m,n}$ let us denote by $(d_1^1)|_f$ the restriction of $d_1^1$ to the fibre at $f$.
    
\begin{lemma}
\label{lem_d1_normal}
Let $a:\Sym^{m-n}\CC^n \to \CC^n$, $b:\Sym^{m-n-1}\to \fsl(\CC^m)$ and $f:\CC^n\to \CC^m$. Then, modulo $\det(\CC^n)^{m-n}$, $(d_1^1)|_f (a,b)=f\circ a + b\circledcirc f$.
\end{lemma}

\begin{proof}
The morphism $d_1^1$ has degree one with respect to $\cO_{M_{m,n}}(-1)$ and is $\GL(m)\times\GL(n)$-equivariant. Thus, modulo a factor $\det(\CC^n)^{m-n}$, $d_1^1$ is given by a $\GL(m)\times\GL(n)$-equivariant linear morphism
\begin{equation}
\label{eq_lin_d_1_normal}
    \begin{aligned}
    \left[ \Sym^{m-n-1}(\CC^n)^\vee \oplus S^{(m-n,0^{n-2},-1)}(\CC^n)^\vee
    \oplus \Hom(\Sym^{m-n-1}\CC^n,\fsl(\CC^m)) \right] \otimes {} \\ \otimes \Hom(\CC^n,\CC^m)  \to \Hom(\Sym^{m-n}\CC^n,\CC^m).
    \end{aligned}
    \end{equation}
    Each one of the three factors $\Sym^{m-n-1}(\CC^n)^\vee \otimes \Hom(\CC^n,\CC^m)$, $S^{(m-n,0^{n-2},-1)}(\CC^n)^\vee \otimes \Hom(\CC^n,\CC^m)$, $\Hom(\Sym^{m-n-1}\CC^n,\fsl(\CC^m)) \otimes \Hom(\CC^n,\CC^m)$ contains a unique copy of the $\GL(m)\times\GL(n)$-representation $\Hom(\Sym^{m-n}\CC^n,\CC^m)$, so there exists a unique non-zero $\GL(m)\times\GL(n)$-equivariant morphism from each of the three factors to $\Hom(\Sym^{m-n}\CC^n,\CC^m)$. The restriction of the morphism $d_1^1$ to each one of the three factors cannot be zero because otherwise such a factor would appear inside $\cF^{W_1}_{-2}$, while $\cF^{W_1}_{-2}$ has constant degree $-r(m-n)-2$ with respect to $\cO_{M_{m,n}}(-1)$ (this can be checked from \cref{prop_cohom_gp}). 
    
    Thus, modulo rescaling (which can always be done on each of the three factors separately), $d_1^1$ is the unique non-zero $\GL(m)\times\GL(n)$-equivariant linear morphism as above. Indeed notice that the morphism defined in the statement is linear and $\GL(m)\times\GL(n)$-equivariant.
\end{proof}

\begin{proposition}
\label{prop_identif_normal}
We have the following identification
$$ (q_*p^* W_1)|_f= \det(\CC^n)^{m-n}\otimes\Hom(\Sym^{m-n}\ker(f),\coker(f)) .$$
\end{proposition}

\begin{proof}
From \cref{thm_res_determ_ext_normal} $q_*p^*W_1$ is the cokernel of $d_1^1:\cF^{W_1}_{-1}\to \cF^{W_1}_0$. Restricting to a fibre is a right exact functor, thus it is sufficient to show that (modulo $\det(\CC^n)^{m-n}$ and a twist by a line bundle) the cokernel of the restriction of \cref{eq_lin_d_1_normal} to the fibre over $f\in \Hom(\CC^n,\CC^m)$ is $\Hom(\Sym^{m-n}\ker(f),\coker(f))$. The image of $(d_1^1)|_f(\Hom(\Sym^{m-n}\CC^n,\CC^n),0)$ is given by all morphisms $\Sym^{m-n}\CC^n\to \CC^m$ whose image is the same as $f$ (in this case the morphism $(d_1^1)|_f$ acts as a change of coordinates in the domain). The image of $(d_1^1)|_f(0,\Hom(\Sym^{m-n-1}\CC^n,\fsl(\CC^m)),0)$ is given by all morphisms $\Sym^{m-n}\CC^n\to \CC^m$ whose kernel contains $\Sym^{m-n}\ker(f)$ (in this case the morphism $(d_1^1)|_f$ acts as a change of coordinates in the codomain). By taking the linear sum of these two subspaces one obtains all morphisms $\eta\in\Hom(\Sym^{m-n}\CC^n, \CC^m)$ such that $\eta(\Sym^{m-n}\ker(f))\subset \Im(f)$, and the cokernel of $(d_1^1)|_f$ is the desired one.
\end{proof}

\section{Relativization of the resolutions to degeneracy loci}
\label{sec_res_deg_locy}

Having obtained \cref{thm_res_determ_ext_normal} for determinantal varieties, we now want to ``relativise'' this result to degeneracy loci of morphisms of vector bundles. We begin by recalling general facts about degeneracy loci of morphisms, with particular focus on the description of their normal bundle. Then we deduce \cref{thm_res_ODL} (which is essentially Theorem B) from \cref{thm_res_determ_ext_normal} in the previous section. We end the section with some applications of \cref{thm_res_ODL} on aCM and Ulrich sheaves, i.e.\ the results announced in Theorem A.

\subsection{On degeneracy loci of morphisms between vector bundles}

Let $X$ be a Cohen--Macaulay variety, $E_1,E_2$ two vector bundles on $X$ of respective ranks $n,m$ with $n\leq m$, and let us fix $r\leq n$. For any section $s\in H^0(X,\Hom(E_1,E_2))$ one can define the following set of points
$$ D_r(s):=\{ x\in X\mid \rank(s(x))\leq r \}\subset X. $$
Clearly $D_n(s)=X$ and $D_0(s)$ is just the zero locus of the section $s$. We have given above the definition of $D_r(s)$ as a set of points but one can also attach a scheme structure to it. These loci, also referred to as \emph{degeneracy loci of morphisms between vector bundles}, enjoy many classically known properties; they can be seen as both generalisations of zero loci of sections of vector bundles and particular cases of \emph{orbital degeneracy loci}, defined and investigated in \cite{BFMT, BFMT2}, to which we refer for the following properties.
\begin{itemize}[leftmargin=*]
    \item[i)] If $D_r(s)$ has the expected codimension $(m-r)(n-r)$ inside $X$ and is contained in the smooth locus $X_{\sm}$ of $X$, then it is Cohen--Macaulay; if $n=m$ it is moreover Gorenstein.
    \item[ii)] The singular locus of $D_r(s)$ contains $D_{r-1}(s)$.
    \item[iii)] If $D_r(s)$ has the expected codimension, there exists a $\cO_X$-locally free resolution $\cF_\bullet $ of $\cO_{D_r(s)}$. If $r=n-1$, one gets the well-known Eagon--Northcott complex (\cite{EN}).
    \item[v)] The section $s$ defines a morphism of vector bundles $E_1\to E_2$ which is generically of maximal rank. As in the affine situation, we get two exact sequences
    $$
0\to E_1\to E_2 \to \coker(s)\to 0.
$$
and
$$
    0\to \coker(s)^\vee\to E_2^\vee\to E_1^\vee \to \cK(s) \to 0,
$$
where $\coker(s)^\vee=\cHom_{\cO_X}(\coker(s),\cO_X)$ and $\cK(s)=\cExt^1_{\cO_X}(\coker(s),\cO_{X})$. When $D_{n-2}(s)$ is empty, $\cK(s)$ is a line bundle supported on $D_{n-1}(s)$ and $\coker(s)|_{D_{n-1}(s)}$ is a vector bundle of rank $m-n+1$ supported on $D_{n-1}(s)$.

Degeneracy loci can be interpreted as the loci of points of $X$ where the morphism, which locally is represented by a matrix, has rank $\leq r$; in other words, a point in $D_r(s)$ is a point where the evaluation of the morphism yields a matrix in the affine subvariety $Y_r$, see \cref{sec_weyman_resolution}. There exist different Kempf collapsings resolving $Y_r$: let us focus on \cref{kempfcoll}. As explained in \cite{BFMT}, interpreting the previous diagram as a fibrewise situation over $X$, where $M_{m,n}$ is a fibre over of $\Hom(E_1,E_2)$, yields an easy way to produce a birational morphism onto $D_r(s)$: let $\Gr(r,E_2)\to X$ be the Grassmann bundle over $X$, and let $\cU,\cQ$ be the relative rank-$r$ and rank-$(m-r)$ tautological bundles on $\Gr(r,E_2)$ satisfying the short exact sequence $0\to \cU \to E_2 \to \cQ \to 0$ (by abuse of notation, pull-back of bundles from $X$ to $\Gr(r,E_2)$ will be denoted by the same notation when no confusion arises). The section $s$ defines by pull-back a section $\overline{s}\in H^0(\Gr(r,E_2),\Hom(E_1,E_2))$, and by passing to the quotient a section $\tilde{s}\in H^0(\Gr(r,E_2),\Hom(E_1,\cQ))$. Then the zero locus $\cZ(\tilde{s}):=D_0(\tilde{s})$ of $\tilde{s}$ is birational to $D_r(s)$, where the morphism $\cZ(\tilde{s})\to D_0(s)$ is given by the restriction of the projection $\Gr(r,E_2)\to X$.
\end{itemize}

\begin{proposition}
\label{prop_normal_Kcoker_fibre}
The normal bundle of $Y:=Y_{n-1}\setminus Y_{n-2}$ inside $M_{m,n}$ is equal to $(K\otimes \coKer)|_Y$.
\end{proposition}

\begin{proof}
We observe that another Kempf collapsing, similar to the one in \cref{kempfcoll} and resolving $Y_r$ for $r=n-1$ is given by

\begin{equation}
\label{koszulComplex2}
\xymatrix{
\rule{10pt}{0pt} \Tot(\cQ_1^\vee\otimes\cU^{}_{2}) \ar@{^{(}->}[dr] \ar@/^/@<1ex>[rrd]^-{{\rho}} \ar@/_/@<-1ex>[drd]_-{{\pi}}\\
& M_{m,n}\times \PP^{n-1}\times\Gr(r,m) \ar[r]_-{\tilde{\rho}} \ar[d]^-{\tilde{\pi}} & \PP^{n-1}\times \Gr(r,m) \\
& M_{m,n}
}
\end{equation}
Here we denoted by:
\begin{itemize}
    \item[$\_$] $\cU_1$ (respectively $\cQ_1$) is the rank one (resp.\ $n-1$) tautological (resp.\ quotient tautological) vector bundle on $\PP^{n-1}$; 
    \item[$\_$] $\cU_2$ (respectively $\cQ_2$) is the rank $n-1$ (resp.\ $m-n+1$) tautological (resp.\ quotient tautological) vector bundle on $\Gr(n-1,m)$.
\end{itemize}

This resolution is an isomorphism on $Y$, which is smooth because it is a single $\GL(m)\times \GL(n)$-orbit. By definition of normal bundle, we have an exact sequence
$$ 
0\to T_{Y} \to T_{M_{m,n}}|_Y \to \cN_{Y/M_{m,n}} \to 0.
$$
Let us consider the pullback $\tilde{\pi}^*t_\id$ to $\tilde{\pi}:\PP^{n-1}\times \Gr(n-1,m)\times M_{m,n} \to M_{m,n}$ of $t_\id$. By passing to the quotient, $\tilde{\pi}^*t_\id$ defines a section $\overline{\tilde{\pi}^*t_\id}$ of the vector bundle
$$ \cN:= \Hom(\CC^n,\CC^m)\otimes \cO / \cQ_{1}^\vee\otimes \cU_2. $$
Then $\Tot(\cQ_{1}^\vee\otimes \cU_2)$ can be seen as the zero locus of $\overline{\tilde{\pi}^*t_\id}$. As a consequence, the normal bundle of $Y$ inside $Z:=\PP^{n-1}\times \Gr(n-1,m)\times M_{m,n}$ is equal to $\cN_{Y/Z}=\cN|_Y$. Thus we have an exact sequence
$$ 
0\to T_{Y} \to (T_{\PP^{n-1}}\times T_{\Gr(n-1,m)}\times T_{M_{m,n}} )|_Y\to \cN|_Y \to 0.
$$
This implies that we have an exact sequence
$$ 0 \to (T_{\PP^{n-1}}\times T_{\Gr(n-1,m)})|_Y \to \cN|_Y \to \cN_{Y/M_{m,n}} \to 0. $$
The construction and the exact sequences are clearly $\GL(m)\times \GL(n)$-equivariant. Notice that we are identifying $Y$ with the subvariety in $Z$ consisting of triples $([V_1],[W_{n-1}],s)\in Z$ such that $s(V_1)=0$ and $s\in \Hom(\CC^{n}/V_1, W_{n-1})$ is an isomorphism. This implies that $\overline{t_\id}$ induces (at each point and thus globally) an isomorphism of vector bundles $\cQ_1|_Y \cong \cU_2|_Y$. Via this isomorphism, $T_{\PP^{n-1}}|_Y= (\cU_1^\vee \otimes \cQ_1)|_Y \cong (\cU_1^\vee \otimes \cU_2)|_Y \subset \cN|_Y$ and $T_{\Gr(n-1,m)}|_Y= (\cU_2^\vee \otimes \cQ_2)|_Y \cong (\cQ_1^\vee \otimes \cQ_2)|_Y \subset \cN|_Y$. From this, it is straightforward to see that 
\[
\cN_{Y/M_{m,n}} = (\cN / T_{\PP^{n-1}}\times T_{\Gr(n-1,m)})|_Y = (\cU_1^\vee \otimes \cQ_2)|_Y= (K\otimes \coKer)|_Y.\qedhere
\]

\end{proof}

\begin{proposition}
\label{prop_normal_Kcoker_relative}
If $D:=D_{n-1}(s)\setminus D_{n-2}(s)$ has expected codimension, the normal bundle of its smooth locus $D_\sm$ inside $X$ is equal to $(\cK(s)\otimes \coker(s))|_{D_\sm}$.  Its exterior powers are given by the formula $$ \wedge^j\cN_{D_\sm/X}= \Sym^j\cK(s)\otimes \wedge^j\coker(s)|_{D_\sm}= \cK(s)^{\otimes j}\otimes \wedge^j\coker(s)|_{D_\sm}.$$
\end{proposition}

\begin{proof}
In the proof of \cref{prop_normal_Kcoker_fibre} we have used the Kempf collapsing \cref{koszulComplex2} to resolve the singularities of $Y_{n-1}$. The relativization of this Kempf collapsing yields a birational morphism onto $D_{n-1}(s)$ from the zero locus of a section of a vector bundle $\cE$ over $\PP(E_1)\times \Gr(n-1,E_2) \to X$, which is an isomorphism over $D$ (cfr.\ once again \cite{BFMT}). Let $Z\to D_\sm$ be the restriction of this birational morphism, which is $1:1$ and thus an isomorphism. $Z$ is an open subset of a zero locus and is contained in the smooth locus of expected codimension, so its normal bundle inside $\PP(E_1)\times \Gr(n-1,E_2)$ is just $\cE|_Z$. Finally notice that both $\cK(s)$ and $\coker(s)$ are vector bundles on $D_\sm$ and can be identified respectively with $\cK(s)=\cU_1^\vee|_{D_\sm}$ and $\coker(s)=\cQ_2|_{D_\sm}$ via $Z\cong D_\sm$.

Now the proof goes exactly as the proof of \cref{prop_normal_Kcoker_fibre}. From \cite{BFMT} we know that $\cE=\Hom(E_1,E_2)\otimes \cO / \cQ_1^\vee \otimes \cU_2$, where $\cQ_1$ and $\cU_2$ are the relative quotient tautological bundle of $\PP(E_1)$ and the relative tautological bundle of $\Gr(n-1,E_2)$, and $\cO$ is the structure sheaf of $\PP(E_1)\times \Gr(n-1,E_2)$. From the isomorphism $Z\cong D_\sm$ we get two normal exact sequences (the former from $D_\sm\subset X$ and the latter from $Z\subset \PP(E_1)\times T_{\Gr(n-1,E_2)}$):
$$ 
0\to T_{D_\sm} \to T_X|_{D_\sm} \to \cN_{D_\sm/X} \to 0 
$$
and
$$
0\to T_{D_\sm} \to (T_{\PP(E_1)}\times T_{\Gr(n-1,E_2)}\times T_{X} )|_{D_\sm}\to \cE|_{D_\sm} \to 0,
$$
from which we deduce that
\begin{equation}
\label{eq_proof_normal}
    0 \to (T_{\PP(E_1)}\times T_{\Gr(n-1,E_2)})|_{D_\sm} \to \cE|_{D_\sm} \to \cN_{D_\sm/X} \to 0.
\end{equation}

Once again, on $Z$ we can identify $\cQ_1|_Z\cong\cU_2|_Z$, which yields $T_{\PP(E_1)}|_{D_\sm}= (\cU_1^\vee \otimes \cQ_1)|_{D_\sm} \cong (\cU_1^\vee \otimes \cU_2)|_{D_\sm} \subset \cE|_{D_\sm}$ and $T_{\Gr(n-1,E_2)}|_{D_\sm}= (\cU_2^\vee \otimes \cQ_2)|_{D_\sm} \cong (\cQ_1^\vee \otimes \cQ_2)|_{D_\sm} \subset \cE|_{D_\sm}$. These two identifications together with \cref{eq_proof_normal} give 
$$\cN_{{D_\sm}/X} = (\cE / T_{\PP(E_1)}\times T_{\Gr(n-1,E_2)})|_{D_\sm} = (\cU_1^\vee \otimes \cQ_2)|_{D_\sm}= (\cK(s)\otimes \coker(s))|_{D_\sm}.$$
The statement on the exterior powers follows from the fact that $\cK(s)$ is a line bundle.
\end{proof}

\begin{remark}
We remark that previous results in \cite{KMR,otherKMR} yield \cref{prop_normal_Kcoker_relative} for $j=1$ in the case of standard determinantal varieties (see \cref{subsect_determinantal}): in this case the normal sheaf of $D_{n-1}(s)$ coincides with $\cK(s)\otimes \coker(s)$ provided some extra hypotheses on the depth of the ideal of submaximal minors.
\end{remark}

From now on we will keep the notation of the previous proposition: $D$ will denote $D_{n-1}(s)\setminus D_{n-2}(s)$ and $D_\sm$ will denote its smooth locus.

\subsection{The resolutions pass to degeneracy loci}

Let us recall the following classical result.

\begin{theorem}[Generic Perfection Theorem (\cite{EN2})]
\label{thm_gen_perf}
Let $S$ be a commutative ring and
$R$ be a polynomial ring over $S$. Let $G_\bullet$ be a free $R$-resolution of a module $M$
of length $l = \depth\Ann_R(M )$ and assume that $M$ is free as an $S$-module.
Let $\psi : R \to R'$ be a ring homomorphism such that $l = \depth \Ann_{R'} (M \otimes_R
R')$. Then $G_\bullet \otimes_R R'$ is a free $R'$-resolution of $M \otimes_R R'$.
\end{theorem}

Let us use this result to relativise the resolution in \cref{thm_res_determ_ext_normal} in order to obtain a resolution supported on degeneracy loci of morphisms.

\begin{theorem}
\label{thm_res_ODL}
Let $E_1$, $E_2$ be vector bundles of rank $n,m$ on a Cohen--Macaulay scheme $X$, $r=n-1$ and $s\in H^0(E_1^\vee \otimes E_2)$ a global section such that $D_r(s)$ has expected pure dimension (i.e.\ pure codimension $m-r$), then for $0 \leq j \leq m-r$ we obtain a locally free resolution $$ 0\to \cE^j_\bullet \to \cW_j:=\Sym^{m-r-j}\cK(s) \otimes \wedge^j \coker(s) \to 0 ,$$ 
where, for $0\leq u\leq m-r$,
$$
\begin{aligned}
\cE_{-u}^{j} = \bigoplus_{\lambda\in \cI_j,|\lambda|=u+r(m-r-j)-j}\tilde{c}_{j\lambda}^\mu S^{\mu^t}E_1 \otimes S^{\tilde{\lambda}}E_2^\vee.
\end{aligned}
$$
\end{theorem}

\begin{proof}
The proof is exactly as in \cite[Theorem 2.5]{BFMT2}. With respect to that proof, we just need to check that we can apply the Generic Perfection Theorem \ref{thm_gen_perf}. By \cref{prop_ann} we know that $\depth \Ann_{\cO_{M_{m,n}}}q_*p^*W_j=\codim_{M_{m,n}}(Y_r)=m-r$, which is equal to the length of the resolution of $q_*p^*W_j$ in \cref{thm_res_determ_ext_normal}. Thus we just need to check that $m-r= \depth \Ann_{\cO_{X,x}} (\cK(s) \otimes \coker(s))_x$ for $x\in X$ (which corresponds to the hypothesis $l = \depth \Ann_{R'} (M \otimes_R
R')$ in \cref{thm_gen_perf}). However, $(\cI_{D_r(s)})_x\subset\Ann_{\cO_{X,x}} (\cK(s) \otimes \coker(s))_x$ again because of \cref{prop_ann}. Moreover, the support of $\cK(s)\otimes \coker(s)$ is set-theoretically $D_r(s)$, so the dimension of $\Ann_{\cO_{X,x}} (\cK(s) \otimes \coker(s))_x$ must be equal to the dimension of $D_r(s)$. By the Auslander--Buchsbaum formula we deduce that $\depth \Ann_{\cO_{X,x}} (\cK(s) \otimes \coker(s))_x = \codim_X(D_r(s))=m-r$. Using this computation the proof in \cite[Theorem 2.5]{BFMT2} works in this situation as well.
\end{proof}

\begin{remark}
We recover the main result in \cite{KMR} when $j=1$. We also recover the Eagon--Northcott complex when $j=m-r$.
\end{remark}

\subsection{Consequences and applications: degeneracy loci}
\label{sec_applications}
We list here some direct consequences of \cref{thm_res_ODL}. Recall that $\cK(s)|_{D_\sm}$ is a line bundle.

\begin{theorem}
\label{thm_loc_CM}
Let $D_r(s)\subset X$ be as in \cref{thm_res_ODL}. Then $\Sym^{m-r-j}\cK(s) \otimes \wedge^j \coker(s)$ is a maximal Cohen--Macaulay $\cO_{D_r(s)}$-sheaf.
\end{theorem}

\begin{proof}
By \cref{thm_res_ODL} and the Auslander--Buchsbaum formula, the depth of the stalk $(\cW_j)_x$ at any closed point $x\in X$ is equal to $\dim(X)-m+r$, which is by hypothesis the dimension of $D_r(s)$. This yields that $(\cW_j)_x$ is a Cohen--Macaulay $\cO_{X,x}$-module. Let $\mathcal{M}_x$ be the maximal ideal of $\cO_{X,x}$. The depth of $(\cW_j)_x$ is the least value of $i$ such that the local cohomology $H^i_{\mathcal{M}_x}((\cW_j)_x)$ is non-zero, see e.g.\ \cite[Definition 4.1]{HunekeTaylor}. By \cite[Proposition 2.14]{HunekeTaylor}
\[
H^i_{\mathcal{M}_x}((\cW_j)_x) = H^i_{\mathcal{M}_x \cO_{D_r(s),x}}((\cW_j)_x),
\]
whence the depths of $(\cW_j)_x$ as an $\cO_{D_r(s),x}$-module and as an $\cO_{X,x}$-module coincide, and they both coincide with $\dim D_r(s)$.
\end{proof}

\begin{corollary}
Let $D_r(s)$ be as in \cref{thm_res_ODL}. Then $\wedge^j \cN_{D_\sm/X}$ is locally Cohen--Macaulay.
\end{corollary}

\begin{proof}
In this case $\wedge^j\cN_{D_\sm/X}= \Sym^j\cK(s)\otimes \wedge^j\coker(s) \otimes \cO_{D_\sm}$, so $\cW_j=\cK(s)^{\otimes m-r-2j}\otimes \wedge^j \cN_{D_\sm/X}$. Being locally Cohen--Macaulay does not depend on twisting by a line bundle, so the result follows by \cref{thm_loc_CM}.
\end{proof}

\subsection{Consequences and applications: determinantal schemes}
\label{subsect_determinantal}

Recall that a \emph{standard determinantal locus} is a locus $D_r(s)$ defined by $s:E_1\to E_2$ over $X$ where $E_1$ and $E_2$ are direct sums of line bundles and $X$ is a projective space. 

\begin{theorem}
\label{thm_sym_wedge_ACM}
With the hypotheses of \cref{thm_res_ODL}, let $D_r(s)\subset \PP^l$ be a standard determinantal locus. Then $\cW_j=\Sym^{m-r-j}\cK(s) \otimes \wedge^j \coker(s)$ is an aCM sheaf for $0\leq j \leq m-r$.
\end{theorem}

\begin{proof}
By \cref{thm_loc_CM}, we just need to check that $H^d(D_r(s),\cW_j(m))=0$ for $0< d < \dim(D_r(s))$. All terms $\cE^j_{-u}$ from \cref{thm_res_ODL} are direct sums of twists of $\cO_X$, and the same holds when we twist $\cE^j_{-u}$ by $\cO_{\PP^l}(m)$. Since $\cO_{\PP^l}$ is aCM, $\cE^j_{-u}(m)$ has cohomology in degree $0$ or $l$; by a spectral sequence argument, this implies that $\cW_j(m)$ has cohomology at most in degree $0$ or $l-u$. The result follows by noticing that $0\leq u\leq m-r=l-\dim(D_r(s))$.
\end{proof}

\begin{corollary}
With the hypotheses of \cref{thm_res_ODL}, let $D_r(s)\subset \PP^l$ be a standard determinantal locus and assume $D_r(s)=D_\sm$. Then $\wedge^j\cN_{D_r(s)/X} \otimes \Sym^{m-r-2j}\cK(s)|_{D_\sm}$ is an aCM sheaf for $0\leq j \leq m-r$.
\end{corollary}

\begin{remark}
\cref{thm_sym_wedge_ACM} in the special case $j=1$ was already covered by some results of \cite{KMR}. More precisely, let $R$ be the polynomial ring of $\PP^l$, $I$ the homogeneous ideal of maximal minors of the matrix defined by $s$, and $J$ the ideal of submaximal minors and assume that $\depth_J(R/I)\geq 4$ (milder assumptions on this depth are actually enough in some cases, see \cite[Prop 3.5]{KMR} and \cite[Thm 3.11]{otherKMR}). Let $K(s)$ and $C(s)$ denote the kernel and cokernel graded modules, whose sheafifications are $\cK(s)$ and $\coker(s)$. In \cite{KMR} it is shown that $\Sym^iK(s)\otimes C(s)=\Ext^1_R(K(s),\Sym^iK(s))=\Hom_{R/I}(I/I^2,\Sym^{i-1}K(s))$ for $0\leq i\leq m-n+1$. Moreover the authors provided a $R$-free resolution of such modules; in the case $i=c-1$, being $c$ the codimension of the determinantal locus, the sheafification of the corresponding resolution coincides with our case $j=1$.
\end{remark}

\begin{remark}
\label{rem_almost_conj}
\cref{thm_sym_wedge_ACM,thm_sym_wedge_Ulrich} and the corresponding corollaries are mainly motivated by the study of the case of standard determinantal loci made in \cite{KMR}, where the authors conjectured \cite[Conjecture 3.8]{KMR} that particular graded modules are Cohen--Macaulay (or Ulrich, in the linear determinantal case, see below). The results contained in this section prove their conjecture almost completely: our methods allow us to solve the conjecture only for the sheafification of these graded modules. 

However, one can recover the corresponding result about the graded modules and hence \cite[Conjecture 3.8]{KMR} for some values of $j$ (more precisely for $0 \leq j \leq (c+1)/2$, $c$ being the codimension of the determinantal scheme) thanks to \cite[Theorem 3.16]{KMR}. As pointed out in \cite[Remark 3.22, arXiv version]{KMR}, the analogous of \cite[Theorem 3.16]{KMR} should hold for any $j$, which would prove the conjecture completely.
\end{remark}

Recall that a \emph{linear determinantal locus} $D_r(s)$ is one defined by $s:\cO(-1)^{\oplus n} \to \cO^{\oplus m}$ over $X=\PP^l$. 

\begin{theorem}
\label{thm_sym_wedge_Ulrich}
With the hypotheses of \cref{thm_res_ODL}, let $D_r(s)$ be a linear determinantal locus. Then $\cW_j((n-1)(m-n))=\Sym^{m-r-j}\cK(s) \otimes \wedge^j \coker(s) \otimes \cO_{\PP^l}((n-1)(m-n))$ is an Ulrich sheaf for $0\leq j \leq m-r$.
\end{theorem}

\begin{proof}
The resolution $\cE^j_\bullet((n-1)(m-n))$ of $\cW_j((n-1)(m-n))$ is linear of length $m-r=\codim(\cW_j((n-1)(m-n)))$.
\end{proof}

\begin{corollary}
With the hypotheses of \cref{thm_res_ODL}, let $D_r(s)\subset X=\PP^l$ be a linear determinantal locus and assume $D_r(s)=D_\sm$. Then $\wedge^j\cN_{D_r(s)/X} \otimes \Sym^{m-r-2j}\cK(s)|_{D_\sm} \otimes \cO_{\PP^l}((n-1)(m-n))|_{D_\sm}$ is Ulrich and $\mu$-semistable for $0\leq j \leq m-r$.
\end{corollary}

\begin{proof}
The only non-trivial statement is the $\mu$-semistability, which is a consequence of \cite[Theorem 2.9]{CHGS}.
\end{proof}

\frenchspacing
\hypersetup{urlcolor=black}

\newcommand{\etalchar}[1]{$^{#1}$}

\end{document}